\documentclass[review]{elsarticle}
\makeatletter
\def\ps@pprintTitle{%
	\let\@oddhead\@empty
	\let\@evenhead\@empty
	\def\@oddfoot{\centerline{\thepage}}%
	\let\@evenfoot\@oddfoot}
\makeatother
\usepackage{lineno,hyperref}
\modulolinenumbers[5]

\usepackage{amsmath,amssymb,epsfig,amsthm,bm}
\usepackage{a4wide,xcolor,eucal,enumerate,mathrsfs,graphics,caption,subfig,booktabs,verbatim}
\usepackage{dsfont,natbib}

\theoremstyle{definition}
\newtheorem{de}{Definition}

\theoremstyle{plain}
\newtheorem{theo}[de]{Theorem}
\newtheorem{lemma}[de]{Lemma}

\theoremstyle{remark}

\newcommand{\R}{\mathbb{R}}

\newcommand{\p}{\mathbb{P}}

\newcommand{\E}{\mathbb{E}}

\newcommand{\1}{\mathds{1}}

\newcommand{\dd}{\mathrm{d}}

\def\argmax{\mathop{\mathrm{argmax}}}

\newcommand{\CM}{\mathrm{CM}}
\journal{Statistics and Probability Letters}

\biboptions{authoryear}

\begin{document}

\begin{frontmatter}

\title{The distance between a naive cumulative estimator and its least concave majorant}


\author[mymainaddress]{Hendrik P.~Lopuha\"a}

\author[mymainaddress]{Eni Musta\corref{mycorrespondingauthor}}
\cortext[mycorrespondingauthor]{Corresponding author}
\ead{e.musta@tudelft.nl}

\address[mymainaddress]{DIAM, Faculty EEMCS, Delft University of Technology, Mekelweg 4, 2628 CD Delft, The Netherlands}

\begin{abstract}
We consider the process $\widehat\Lambda_n-\Lambda_n$, where $\Lambda_n$ is a cadlag step estimator for the primitive $\Lambda$
of a nonincreasing function $\lambda$ on $[0,1]$, and $\widehat\Lambda_n$ is the least concave majorant of $\Lambda_n$.
We extend the results in~\cite{kulikov-lopuhaa2006SPL,kulikov-lopuhaa2008} to the general setting considered in~\cite{durot2007}.
Under this setting we prove that a suitably scaled version of $\widehat\Lambda_n-\Lambda_n$
converges in distribution to the corresponding process for two-sided Brownian motion with parabolic drift
and we establish a central limit theorem for the $L_p$-distance between $\widehat\Lambda_n$ and $\Lambda_n$.
\end{abstract}


\end{frontmatter}


\section{Introduction}
\label{sec:introduction}
Grenander-type estimators are well known methods for estimation of monotone curves.
In case of estimating nonincreasing curves, they are constructed by starting with a naive estimator for the primitive of the curve of interest
and then take the left-derivative of the least concave majorant (LCM) of the naive estimator.
The first example can be found in~\cite{grenander1956} in the context of estimating a nonincreasing density $f$ on $[0,\infty)$
on the basis of an i.i.d.~sample from $f$.
The empirical distribution function $F_n$ of the sample is taken as a naive estimator for the cumulative distribution function
corresponding to $f$ and the Grenander estimator is found by taking the left-derivative $\widehat{f}_n$ of the least concave majorant $\widehat{F}_n$.
Similar estimators have been developed in other statistical models,
e.g., regression (see~\cite{brunk1958}), 
random censoring (see ~\cite{huang-wellner1995}), 
or the Cox model (see~\cite{LopuhaaNane2013}).
\cite{durot2007} considers Grenander-type estimators in a general setup that incorporates several statistical models.
A large part of the literature is devoted to investigating properties of Grenander-type estimators for monotone curves,
and somewhat less attention is paid to properties of the difference between the corresponding naive estimator for the primitive of the curve 
and its LCM.

\cite{kiefer-wolfowitz1976} show that $\sup_t|\widehat F_n-F_n|=O_p((n^{-1}\log n)^{2/3})$.
Although the first motivation for this type of result has been asymptotic optimality of shape constrained estimators,
it has several important statistical applications.
The Kiefer-Wolfowitz result was a key argument in~\cite{SenBodhisattvaBanerjeeWoodroofe2010} to prove that the $m$ out of $n$ bootstrap from $\widehat F_n$ works.
\cite{mammen1991} suggested to use the result to make an asymptotic comparison between
a smoothed Grenander-type estimator and an isotonized kernel estimator in the regression context. 
See also~\cite{WangWoodroofe2007} for a similar application of their Kiefer-Wolfowitz comparison theorem.
An extension to a more general setting was established in~\cite{DurotLopuhaa14},
which has direct applications in~\cite{DurotGroeneboomLopuhaa2013} to prove that a smoothed bootstrap from a Grenander-type estimator
works for $k$-sample tests, and in~\cite{GroeneboomJongbloed13} and~\cite{LopuhaaMusta2017} to extract the pointwise limit behavior of smoothed Grenander-type estimators for a monotone hazard
from that of ordinary kernel estimators.
To approximate the $L_p$-error of smoothed Grenander-type estimators by that of ordinary kernel estimators,
such as in~\cite{CsorgoHorvath1988} for kernel density estimators, a Kiefer-Wolfowitz type result no longer suffices.
In that case, results on the $L_p$-distance, between $\widehat F_n$ and $F_n$ are more appropriate,
such as the ones in~\cite{durottoquet2003} and~\cite{kulikov-lopuhaa2008}.

In this paper, we extend the results in~\cite{durottoquet2003} and~\cite{kulikov-lopuhaa2008} to the general setting of~\cite{durot2007}.
Our main result is a central limit theorem for the $L_p$-distance between
$\widehat{\Lambda}_n$ and $\Lambda_n$, where $\Lambda_n$ is a naive estimator for the primitive $\Lambda$ of a monotone curve $\lambda$ 
and $\widehat\Lambda_n$ is the LCM of $\Lambda_n$.
As special cases we recover Theorem~5.2 in~\cite{durottoquet2003} and Theorem~2.1 in~\cite{kulikov-lopuhaa2008}.
Our approach requires another preliminary result, which might be of interest in itself, i.e.,
a limit process for a suitably scaled difference between~$\widehat\Lambda_n$ and $\Lambda_n$.
As special cases we recover Theorem~1 in~\cite{Wang1994}, Theorem~4.1 in~\cite{durottoquet2003},
and Theorem~1.1 in~\cite{kulikov-lopuhaa2006SPL}.

\section{Main results} 
We consider the general setting in~\cite{durot2007}.
Let $\lambda:[0,1]\to\R$ be nonincreasing and assume that we have at hand a cadlag step estimator $\Lambda_n$ of
\[
\Lambda(t)=\int_0^t\lambda(u)\,\dd u, \quad t\in[0,1].
\]
In the sequel we will make use of the following assumptions.
\begin{enumerate}
\item[(A1)]
$\lambda$ is strictly decreasing and twice continuously differentiable on $[0,1]$ with $\inf_t|\lambda'(t)|>0$.
\item[(A2)]
Let $B_n$ be either a Brownian motion or a Brownian bridge.
There exists $q>6$, $C_q>0$, $L:[0,1]\to\R$, and versions of $M_n=\Lambda_n-\Lambda$ and $B_n$ such that
\[
\p\left(n^{1-1/q}\sup_{t\in[0,1]}\left|M_n(t)-n^{-1/2}B_n\circ L(t)\right|>x \right)\leq C_qx^{-q}
\]
for all $x\in(0,n]$. 
Moreover, $L$ is increasing and twice differentiable on $[0,1]$, with $\sup_t|L''(t)|<\infty$ and $\inf_t|L'(t)|>0.$
\end{enumerate}
Note that this setup includes several statistical models,
such as monotone density, monotone regression, and the monotone hazard  model under random censoring,
see~\cite{durot2007}[Section 3].

We consider the distance between $\Lambda_n$ and its least concave majorant $\widehat{\Lambda}_n=\CM_{[0,1]}\Lambda_n$, 
where~$\CM_I$ maps a function $h:\R\to\R$ into the least concave majorant of $h$ on the interval $I\subset\R$.
Consider the process
\begin{equation}
\label{def:An}
A_n(t)=n^{2/3}\left( \widehat{\Lambda}_n(t)-\Lambda_n(t)\right),\qquad t\in[0,1],
\end{equation}
and define
\begin{equation}
\label{def:Z and zeta}
Z(t)=W(t)-t^2,\qquad \zeta(t)=[\CM_{\R}Z](t)-Z(t),
\end{equation}
where $W$ denotes a standard two-sided Brownian motion originating from zero.
For each $t\in(0,1)$ fixed and $t+c_2(t)sn^{-1/3}\in(0,1)$, define
\begin{equation}
\label{def:zeta nt}
\zeta_{nt}(s)=c_1(t)A_n\left(t+c_2(t)sn^{-1/3}\right),
\end{equation}
where
\begin{equation}
\label{def:c1c2}
c_1(t)=\left(\frac{|\lambda'(t)|}{2L'(t)^2} \right)^{1/3},\qquad c_2(t)=\left(\frac{4L'(t)}{|\lambda'(t)|^2} \right)^{1/3}.
\end{equation}
Our first result is the following theorem, {which extends Theorem~1.1 in~\cite{kulikov-lopuhaa2006SPL}}.
\begin{theo}
\label{theo:limit process}
Suppose that assumptions (A1)-(A2) are satisfied.
Let $\zeta_{nt}$ and $\zeta$ be defined in~\eqref{def:zeta nt} and~\eqref{def:Z and zeta}.
Then the process $\{\zeta_{nt}(s):s\in\R\}$ converges in distribution to the process
$\{\zeta(s):s\in\R\}$ in $D(\R)$, the space of cadlag function on $\R$.
\end{theo}
Note that as a particular case $\zeta_{nt}(0)$ converges weakly to $\zeta(0)$.
In this way, we recover Theorem~1 in~\cite{Wang1994} and Theorem~4.1 in~\cite{durottoquet2003}.
The proof of Theorem~\ref{theo:limit process} follows the line of reasoning in~\cite{kulikov-lopuhaa2006SPL}. 

{Let us briefly sketch the argument to prove Theorem~\ref{theo:limit process}.	
Note that $A_n=D_{[0,1]}[n^{2/3}\Lambda_n]$ and $\zeta=D_{\R}[Z]$, where
$D_Ih=\CM_Ih-h$, for $h:\R\to\R$.
Since $D_I$ is a continuous mapping, 
the main idea is to apply the continuous mapping theorem to properly scaled approximations
of the processes~$\Lambda_n$ and $Z$ on a suitable chosen fixed interval $I$.
The first step is to determine the weak limit of~$\Lambda_n$, 
which is given in the following lemma}. 
\begin{lemma}
	\label{lem:Lemma1.1SPL}
	Suppose that assumptions (A1)-(A2) are satisfied.
	Then for $t\in(0,1)$ fixed, the process
	$X_{nt}(s)
	=
	n^{2/3}
	\left(
	\Lambda_n(t+sn^{-1/3})-\Lambda_n(t)-\left(\Lambda(t+sn^{-1/3})-\Lambda(t)\right)
	\right)$
	converges in distribution to the process $\{W(L'(t)s):s\in\R\}$.
\end{lemma}
{Since $n^{2/3}(\Lambda(t+sn^{-1/3})-\Lambda(t))\approx n^{1/3}\lambda(t)s+\lambda'(t)s^2/2$
and $D_I$ is invariant under addition of linear functions,
it follows that the process $A_n$ can be approximated by a Brownian motion with a parabolic drift.
The idea now is to use continuity of $D_I$, for a suitably chosen interval~$I=[-d,d]$,
to show that~$D_{I}E_{nt}$  converges to $D_I Z_t$,
where
\begin{equation}
\label{def:Zt}
\begin{split}
E_{nt}(s)
&=
n^{2/3}\Lambda_n(t+sn^{-1/3})\\
Z_t(s)
&=
W(L'(t)s)+\lambda'(t)s^2/2.
\end{split}
\end{equation}
In order to relate this to the processes $\zeta_{nt}$ and $\zeta$ in Theorem~\ref{theo:limit process},
note that $A_n(t+sn^{-1/3})=[D_{I_{nt}}E_{nt}](s)$, 
where $I_{nt}=[-tn^{1/3},(1-t)n^{1/3}]$,
and by Brownian scaling,
the process $Z(s)$ has the same distribution as the process $c_1(t)Z_t(c_2(t)s)$.
This means that we must compare the concave majorants of $E_{nt}$ on the intervals $I_{nt}$ and $I$,
as well as the concave majorants of $Z_t$ on the interval~$I$ and $\R$.
Lemma~1.2 in~\citet{kulikov-lopuhaa2006SPL} shows that, locally, with high probability, 
both concave majorants of the process $Z_t$ coincide on $[-d/2,d/2]$, for large $d>0$.
A similar result is established for the concave majorants of the process $E_{nt}$
in Lemma~\ref{lemma:probability_N_W}, which is analogous to Lemma~1.3 in~\citet{kulikov-lopuhaa2006SPL}.
As a preparation for Theorem~\ref{theo:L_p_lcm}, the lemma also contains a similar result for a Brownian motion version of $E_{nt}$.}

Let~$B_n$ be as in assumption~(A2) and let $\xi_n$ be a $N(0,1)$ distributed random variable
independent of $B_n$, if $B_n$ is a Brownian bridge, and $\xi_n=0$, when $B_n$ is a Brownian motion.
Define versions $W_n$ of a Brownian motion by $W_n(t)=B_n(t)+\xi_nt$, for $t\in[0,1]$,
and define
\begin{equation}
\label{def:AnW}
A_n^W
=
n^{2/3}
\left(
\CM_{[0,1]}\Lambda^W_n-\Lambda_n^W
\right)
\end{equation}
where $\Lambda_n^W(t)=\Lambda(t)+n^{-1/2}W_n(L(t))$,
with $L$ as in assumption (A2).
Furthermore, define $E_n=\sqrt{n}(\Lambda_n-\Lambda)$, $\Lambda_n^E=\Lambda_n$, $A_n^E=A_n$. 
{The superscripts $E$ and $W$ refer to the empirical and Brownian motion version.
For $d>0$,} let $I_{nt}(d)=[0,1]\cap [t-dn^{-1/3},t+dn^{-1/3}]$ and, for $J=E,W$, define the event
\begin{equation}
\label{def:N^J_nt}
N_{nt}^J(d)
=
\left\{
[\CM_{[0,1]}\Lambda_n^J](s)=[\CM_{I_{nt}(d)}\Lambda_n^J](s),\text{ for all }s\in I_{nt}(d/2) 
\right\}.
\end{equation}
Let $I_{nt}=I_{nt}(\log n)$ and $N^J_{nt}=N^J_{nt}(\log n)$.

\begin{lemma}
	\label{lemma:probability_N_W}
	Assume that assumptions (A1)-(A2) hold.
	For $d>0$, let $N^J_{nt}(d)$ be the event defined in \eqref{def:N^J_nt}.
	There exists $C>0$, independent of $n$, $t$, $d$, such that
	\[
	\begin{split}
	\p\left((N^W_{nt}(d))^c\right)
	&=
	O\left(\mathrm{e}^{-Cd^3} \right)\\
	\p\left((N^E_{nt}(d))^c\right)
	&=
	O\left(n^{1-q/3}d^{-2q}+\mathrm{e}^{-Cd^3} \right),
	\end{split}
	\]
	where $q$ is from assumption (A2).
\end{lemma}
{The proof of Theorem~\ref{theo:limit process} now follows the same line of reasoning as that of Theorem~1.1 in~\citet{kulikov-lopuhaa2006SPL},
	see Section~\ref{sec:proofs} for more details.}
	The next step is to deal with the $L_p$ norm. 
	Our main result is the following.
\begin{theo}
\label{theo:L_p_lcm}
Suppose that assumptions (A1)-(A2) are satisfied and let 
{$A_n$ and~$\zeta$
be defined by~\eqref{def:An} and~\eqref{def:Z and zeta}, respectively.}
Let $\mu$ be a measure on the Borel sets of $\R$, such that
\begin{enumerate}
	\item[(A3)]
	$\dd\mu(t)=w(t)\,\dd t,$ where $w(t)\geq 0$ is differentiable with bounded derivative on $[0,1]$.
\end{enumerate}
Then, for all $1\leq p< \min(q,2q-7)$, (with $q$ as in assumption (A2)),
\[
n^{1/6}\left(\int_0^1 A_n(t)^p\,\dd \mu(t)-m\right)\xrightarrow{d} N(0,\sigma^2),
\]
where
\[
m=\E\left[\zeta(0)^p \right]\int_0^1\frac{2^{p/3}L'(t)^{2p/3}}{|\lambda'(t)|^{p/3}}\,\dd \mu(t)
\]
and
\[
\sigma^2
=
\int_0^1\frac{2^{(2p+5)/3}L'(t)^{(4p+1)/3}}{|\lambda'(t)|^{(2p+2)/3}}w^2(t)\,\dd t
\int_0^\infty \mathrm{cov}\left(\zeta(0)^p,\zeta(s)^p\right)\,\dd s.
\]
\end{theo}
{For the special cases that $\lambda$ is a probability density or a regression function, we recover Theorem~2.1 in~\citet{kulikov-lopuhaa2008}
and Theorem~5.2 in\citet{durottoquet2003}, respectively.}
In order to prove Theorem~\ref{theo:L_p_lcm} we first need some preliminary results. 
We aim at approximating the~$L_p$-norm of $A_n$ by that of the Brownian motion version $A^W_n$ and then finding the asymptotic distribution for the latter one. 
To this end, we first need to relate the moments of $A_n$ to those of~$A_n^W$.
We start by showing that, for $J=E,\,W,$ a rescaled version of $\Lambda^J_n$ can be approximated 
by {the same} process $Y_{nt}$ plus a linear term. This result corresponds to Lemma 4.1 in~\cite{kulikov-lopuhaa2008}.
\begin{lemma}
\label{lemma:4.1}
Suppose that  assumptions (A1)-(A2) are satisfied. 
Then, for $t\in(0,1)$ fixed, for $J=E,W,$ and $s\in[-tn^{1/3},(1-t)n^{1/3}]$, it holds
$n^{2/3}\Lambda_n^J(t+n^{-1/3}s)=Y_{nt}(s)+L_{nt}^J(s)+R_{nt}^J(s)$,
where~$L_{nt}^J(s)$ is linear in $s$ and
$Y_{nt}(s)
=
n^{1/6}
\left\{
W_n(L(t+n^{-1/3}s))-W_n(L(t))
\right\}
+
\frac{1}{2}\lambda'(t)s^2$.
Moreover, for all $p\geq 1$,
\[
\E\left[\sup_{|s|\leq \log n}\left|R_{nt}^W(s) \right|^p \right]
=
O\left(n^{-p/3}(\log n)^{3p} \right),
\]
uniformly in $t\in(0,1)$.
If, in addition $1\leq p<q$ (with $q$ as in assumption (A2)), then
\[
\E\left[\sup_{|s|\leq \log n}\left|R_{nt}^E(s) \right|^p \right]
=O\left(n^{-p/3+p/q}\right)
\]
uniformly in $t\in(0,1)$.
\end{lemma}
Since the map $D_I$ is invariant under addition of linear terms, Lemma~\ref{lemma:4.1} allows us to approximate the moments of $A^J_n(t)=n^{2/3}D_{[0,1]}\Lambda^J_n$ 
by those of $[D_{H_{nt}}Y_{nt}](0)$ for some interval $H_{nt}$, as in Lemma~4.2 in~\cite{kulikov-lopuhaa2008}.
\begin{lemma}
\label{lemma:4.2}
Suppose that  assumptions (A1)-(A2) are satisfied.
{and let $Y_{nt}$ be the process defined in Lemma~\ref{lemma:4.1}.}
Define $H_{nt}=[-n^{1/3}t,n^{1/3}(1-t)]\cap[-\log n,\log n]$.
Then for all $p\geq 1$, it holds
\[
\E\left[A^W_n(t)^p\right]
=
\E\left[\left[\mathrm{D}_{H_{nt}}Y_{nt}\right](0)^p\right]+o\left(n^{-1/6}\right),
\]
uniformly for $t\in(0,1)$.
If, in addition $1\leq p<\min(q,2q-7)$, with $q$ from condition (A2),
then also
\[
\E\left[A^E_n(t)^p\right]=\E\left[\left[\mathrm{D}_{H_{nt}}Y_{nt}\right](0)^p\right]+o\left(n^{-1/6}\right),
\]
uniformly for $t\in(0,1)$.
\end{lemma}
The process $Y_{nt}$ has the same distribution as 
\begin{equation}
\label{def:tilde_Y}
\tilde{Y}_{nt}=W\left(n^{1/3}\left(L\left(t+n^{-1/3}s\right)-L(t)\right)\right)+\frac{1}{2}\lambda'(t)s^2,
\end{equation}
which is close to the process {$Z_t$ in~\eqref{def:Zt}} by continuity of Brownian motion. 
Lemma 4.3 in~\cite{kulikov-lopuhaa2008} is then  used to  show that {the} concave majorants at zero are sufficiently close. 
Note that, with by Brownian scaling, the process $c_1(t)Z_t(c_2(t)s)$ has the same distribution as the process $Z(s)$. 
As a consequence of Lemma~\ref{lemma:4.2} the moments of $A^J_n(t)$ can be related to those of the process $\zeta$.
This formulated in the next lemma, which corresponds to Lemma~4.4 in~\citet{kulikov-lopuhaa2008}.
\begin{lemma}
\label{lemma:4.4}
Suppose that  assumptions (A1)-(A2) are satisfied.
Then, for all $p\ge1$,
\[
\E\left[A_n^W(t)^p\right]
=
\left(\frac{2L'(t)^2}{|\lambda'(t)|}\right)^{p/3}\E\left[\zeta(0)^p\right]+o\left( n^{-1/6}\right)
\]
uniformly in $t\in(n^{-1/3}\log n,1-n^{-1/3}\log n)$ and
\[
\E\left[A_n^W(t)^p\right]\leq\left(\frac{2L'(t)^2}{|\lambda'(t)|}\right)^{p/3}\E\left[\zeta(0)^p\right]+o\left( n^{-1/6}\right)
\]
uniformly in $t\in(0,1)$.
If, in addition $1\leq p<\min(q,2q-7)$, where $q$ is from assumption (A2), then the same (in)equalities hold for $A_N^E(t)$.
\end{lemma}
In {Lemmas~\ref{lemma:4.2} and~\ref{lemma:4.4}} the moments of $A_n^E$ and $A_n^W$ are approximated by the moments of the same process. 
This suggests that  the difference between them is of smaller order than $n^{-1/6}$. 
Indeed, on the events $N^J_{nt}$, where $A^J_n=n^{2/3}D_{I_{nt}}\Lambda^J_n$, 
we make use of {Lemma~\ref{lemma:4.2}} and the fact that $D_I$ is invariant under addition of linear functions to obtain that
\[
\sup_{t\in(0,1)}\left|n^{2p/3}[D_{I_{nt}}\Lambda^E_n](t)-n^{2p/3}[D_{I_{nt}}\Lambda^W_n](t) \right|\leq \sup_{t\in(0,1)}\sup_{|s|\leq \log n}\left\{|R^E_{nt}(s)|+|R^W_{nt}(s)|\right\},
\]
where the processes $R^J_{nt}$ converge to zero sufficiently fast. 
On the other hand, on $(N^J_{nt})^c$ we just need the boundedness of the moments of $A^J_{n}$, 
which follows by Lemma~\ref{lemma:4.4} and the fact that the probability of these events is very small (Lemma~\ref{lemma:probability_N_W}).

\begin{lemma}
\label{lemma:4.5}
Suppose that  assumptions (A1)-(A2) are satisfied. 
Then, for  $1\leq p<\min(q,2q-7)$, with $q$ from assumption (A2),
it holds
\[
\begin{split}
\E\left[\left|A_n^E(t)^p-A^W_n(t)^p\right| \right]
&=
o\left( n^{-1/6}\right)\\
\E\left[\left|A_n^E(t)-A^W_n(t)\right|^p \right]
&=
o\left( n^{-1/6}\right)
\end{split}
\]
uniformly in $t\in(0,1)$. 
\end{lemma}
From Lemma~\ref{lemma:4.4} it follows that 
$n^{1/6}|m-\int_0^1\E\left[A_n^W(t)^p \right]\,\mathrm{d}t |\to 0$,
where $m$ is the asymptotic mean in Theorem~\ref{theo:L_p_lcm}.
Moreover, Lemma~\ref{lemma:4.5} implies that
\[
n^{1/6}\left|\int_0^1 A_n^E(t)^p \,\mathrm{d}t-\int_0^1 A_n^W(t)^p\,\mathrm{d}t \right|
\leq 
n^{1/6}\int_0^1 \left|A_n^E(t)^p- A_n^W(t)^p\right| \,\mathrm{d}t \to 0.
\]
As a consequence, in order to prove Theorem~\ref{theo:L_p_lcm}, it suffices to prove 
asymptotic normality of its Brownian motion version
\[
T^W_n=n^{1/6}\int_0^1\left(A^W_n(t)^p-\E\left[ A^W_n(t)^p\right] \right)\,\dd \mu(t).
\]
The proof of this is completely similar to that of Theorem~2.1 in~\cite{kulikov-lopuhaa2008}. 
First, by using Theorem~\ref{theo:limit process} for a Brownian version of $\zeta_{nt}$ and the mixing property of $A^W_n$
{(this can be obtained in the same way as Lemma~4.6 in~\cite{kulikov-lopuhaa2008})}, 
we derive the asymptotic variance of $T^W_n$ in the following lemma.
\begin{lemma}
	\label{lemma:variance_brownian_version}
	Suppose that  assumptions (A1)-(A3) are satisfied. 
	Then, for every $p\geq 1$,
	\[
	\mathrm{Var}\left(n^{1/6}\int_0^1 A^W_n(t)^p\,\dd \mu(t) \right)
	\to\int_0^1\frac{2^{(2p+5)/3}L'(t)^{(4p+1)/3}}{|\lambda'(t)|^{(2p+2)/3}}w^2(t)\,\dd t\int_0^\infty \mathrm{cov}\left(\zeta(0)^p,\zeta(s)^p\right)\,\dd s.
	\]
\end{lemma}
{The last step is proving the asymptotic normality of $T^W_n$. 
This is done by a big-blocks small-blocks argument, where the contribution of the small blocks to the asymptotic distribution is negligible, 
while the mixing property of $A^W_n$ allows us to approximate the sum over the big blocks by a sum of independent random variables which satisfy the assumptions of Lindeberg central limit theorem.}
\section{Proofs}
\label{sec:proofs}
\begin{proof}[Proof of Lemma~\ref{lem:Lemma1.1SPL}]
	The proof is completely similar to that of Lemma~1.1 in~\cite{kulikov-lopuhaa2006SPL},
	but this time $E_n=\sqrt{n}(\Lambda_n-\Lambda)$ and
	$\sup_{t\in[0,1]}
	\left|
	E_n(t)-B_n\circ L(t)
	\right|
	=
	O_p(n^{-1/2+1/q})$,
	according to~(A2).
	Similar to the proof of Lemma~1.1 in~\cite{kulikov-lopuhaa2006SPL}, this means that
	\[
	X_{nt}(s)
	=
	n^{1/6}
	\left(
	W_n(L(t+sn^{-1/3}))-W_n(L(t))
	\right)
	+
	O_p(n^{-1/3+1/q})
	\stackrel{d}{=}
	W(L'(t)s)+R_n(s),
	\]
	where $\sup_{s\in I}|R_n(s)|\to0$ in probability for compact $I\subset\R$.
	From here on the proof is the same as that of Lemma~1.1 in~\cite{kulikov-lopuhaa2006SPL}.
\end{proof}

\begin{proof}[Proof Lemma~\ref{lemma:probability_N_W}]
	Let $\widehat{\lambda}^W_n$ be the left derivative of $\widehat{\Lambda}^W_n=\CM_{[0,1]}\Lambda^W_n$.
	Define the inverse process
	\[
	U^W_n(a)=\argmax_{t\in[0,1]}\left\{\Lambda^W_n(t)-at \right\}
	\quad\text{and}\quad
	V^W_n(a)=n^{1/3}\left(L(U^W_n(a))-L(g(a))\right),
	\]
	where $g$ denotes the inverse of $\lambda$.
	As in the proof of Lemma~1.3 in~\cite{kulikov-lopuhaa2006SPL}[see~(2.2)], we get
	\begin{equation}
	\label{eqn:probability_N_W}
	\begin{split}
	\p\left((N^W_{nt}(d))^c\right)
	&\leq
	\p\left(\widehat{\lambda}^W_n(t-n^{-1/3}d)=\widehat{\lambda}^W_n(t-n^{-1/3}d/2)\right)\\
	&\quad+
	\p\left(\widehat{\lambda}^W_n(t+n^{-1/3}d)=\widehat{\lambda}^W_n(t+n^{-1/3}d/2)\right).
	\end{split}
	\end{equation}
	Then, with $s=t-dn^{-1/3}/2$, $x=d/2$, and $\epsilon_n=\inf_{t\in[0,1]} |\lambda'(t)|dn^{-1/3}/8$,
	it holds (see (2.3) in~\cite{kulikov-lopuhaa2006SPL}),
	\begin{equation}
	\label{eqn:probability_N_W2}
	\begin{split}
	\p\left(\widehat{\lambda}^W_n(t-n^{-1/3}d)=\widehat{\lambda}^W_n(t-n^{-1/3}d/2)\right)
	&\leq
	\p\left(\widehat{\lambda}^W_n(s+n^{-1/3}x)-\lambda(s+n^{-1/3}x)>\epsilon_n\right)\\
	&\quad+
	\p\left(\widehat{\lambda}^W_n(s)-\lambda(s)<-\epsilon_n\right).
	\end{split}
	\end{equation}
	Moreover, using the switching relation
	$\widehat{\lambda}^W_n(t)\leq a\Leftrightarrow U^W_n(a)\leq t$,
	we rewrite this probability as
	\[
	\begin{split}
	&\p\left\{U^W_n(\lambda(s+n^{-1/3}x)+\epsilon_n)>s+n^{-1/3}x\right\}\\
	&=\p\left\{V^W_n(\lambda(s+n^{-1/3}x)+\epsilon_n)>n^{1/3}\left(L(s+n^{-1/3}x)-L(g(\lambda(s+n^{-1/3}x)+\epsilon_n)) \right)\right\}\\
	&=\p\left\{V^W_n(\lambda(s+n^{-1/3}x)+\epsilon_n)>\frac{\inf_{t\in[0,1]} |\lambda'(t)|\inf_{t\in[0,1]} L'(t) d}{8\sup_{t\in[0,1]} |\lambda'(t)|}\right\}.
	\end{split}
	\]
	It suffices to show that there exists positive constants $C_1$, $C_2$ such that
	\begin{equation}
	\label{eqn:tail_probabilities}
	\p\left(V^W_n(a)>x \right)\leq C_1\mathrm{e}^{-C_2x^3}
	\end{equation}
	because then it follows that
	\[
	\p\left(V^W_n(\lambda(s+n^{-1/3}x)+\epsilon_n)>\frac{\inf_{t\in[0,1]} |\lambda'(t)|\inf_{t\in[0,1]} L'(t) d}{8\sup_{t\in[0,1]} |\lambda'(t)|}\right)\leq \tilde{C}_1\mathrm{e}^{-\tilde{C}_2d^3}.
	\]
	Similarly we can also bound the second probabilities in \eqref{eqn:probability_N_W} and \eqref{eqn:probability_N_W2}. Then the statement of the lemma follows immediately.
	
	Now we prove \eqref{eqn:tail_probabilities}.
	First write
	\[
	\begin{split}
	V_n^W(a)
	&=
	n^{1/3}
	\left(
	L\left(\argmax_{t\in[0,1]}\Big\{W(L(t))+\sqrt{n}(\Lambda(t)-at)\Big\}\right)-L(g(a))
	\right)\\
	&=
	n^{1/3}
	\left(
	\argmax_{s\in[L(0),L(1)]}\Big\{W(s)+\sqrt{n}(\Lambda\left(L^{-1}(s))-aL^{-1}(s)\right)\Big\}-L(g(a))
	\right).
	\end{split}
	\]
	Using properties of the $\argmax$ functional we obtain that the right hand side is equal to the argmax of the process
	\[
	\begin{split}
	&
	n^{1/6}
	\left\{
	W\left(n^{-1/3}s+L(g(a))\right)-W\left(L(g(a))\right)
	\right\}\\
	&+
	n^{2/3}
	\left\{
	\Lambda
	\left(
	L^{-1}
	\left(
	n^{-1/3}s+L(g(a))
	\right)
	\right)-\Lambda(g(a))
	-
	aL^{-1}
	\left(
	n^{-1/3}s+L(g(a))
	\right)
	+ag(a)
	\right\}
	\end{split}
	\]
	for $s\in I_n(a)=[n^{1/3}(L(0)-L(g(a))),n^{1/3}(L(1)-L(g(a)))]$.
	By Brownian motion scaling,
	$V^W_n(a)$ is equal in distribution to $\argmax_{t\in I_n(a)}\{W(t)-D_{a,n}(t)\}$,
	where $W$ is a standard two-sided Brownian motion originating from zero and
	\[
	\begin{split}
	D_{a,n}(s)
	&=
	-n^{2/3}
	\bigg\{
	\Lambda\left(L^{-1}\left(n^{-1/3}s+L(g(a))\right)\right)-\Lambda(g(a))\\
	&\qquad\qquad\qquad-
	aL^{-1}\left(n^{-1/3}s+L(g(a))\right)+ag(a)
	\bigg\}.
	\end{split}
	\]
	By Taylor's formula and the assumptions on $\lambda$ and $L$, one can show that there exist a constant~$c_0>0$, independent of $n$, $a$ and $t$,
	such that $D_{a,n}(t)\geq c_0 t^2$.
	Then~\eqref{eqn:tail_probabilities} follows from Theorem~4 in~\cite{durot2002}, which proves the first statement.
	
	To continue with the second statement, let $\widehat{\lambda}_n$ be the left derivative of $\widehat{\Lambda}_n$
	and define the inverse process
	\[
	U_n(a)=\argmax_{t\in[0,1]}\left\{\Lambda_n(t)-at \right\},\quad\text{and}\quad V_n(a)=n^{1/3}\left(U_n(a)-g(a)\right),
	\]
	where $g$ denotes the inverse of $\lambda$.
	As in~\eqref{eqn:probability_N_W}, we get
	\begin{equation}
	\label{eqn:probability_N_E}
	\begin{split}
	\p\left((N^E_{nt}(d))^c\right)
	&\leq
	\p\left(\widehat{\lambda}_n(t-n^{-1/3}d)=\widehat{\lambda}_n(t-n^{-1/3}d/2)\right)\\
	&\quad+
	\p\left(\widehat{\lambda}_n(t+n^{-1/3}d)=\widehat{\lambda}_n(t+n^{-1/3}d/2)\right).
	\end{split}
	\end{equation}
	where similar to~\eqref{eqn:probability_N_W2},
	\begin{equation}
	\label{eqn:probability_N_E2}
	\begin{split}
	\p\left(\widehat{\lambda}_n(t-n^{-1/3}d)=\widehat{\lambda}_n(t-n^{-1/3}d/2)\right)
	&\leq
	\p\left(\widehat{\lambda}_n(s+n^{-1/3}x)-\lambda(s+n^{-1/3}x)>\epsilon_n\right)\\
	&\quad+
	\p\left(\widehat{\lambda}_n(s)-\lambda(s)<-\epsilon_n\right).
	\end{split}
	\end{equation}
	Then using the switching relation $\widehat{\lambda}_n(t)\leq a\Leftrightarrow U_n(a)\leq t$,
	we rewrite the first probability in~\eqref{eqn:probability_N_E2} as
	\[
	\p\left(V_n(\lambda(s+n^{-1/3}x)+\epsilon_n)>\frac{\inf_{t\in[0,1]} |\lambda'(t)|d}{8\sup_{t\in[0,1]} |\lambda'(t)|}\right).
	\]
	According to Lemma 6.4 in \cite{DurotKulikovLopuhaa2012}, there exists positive constants $C_1, C_2>0$,
	independent of $n$, $a$, and $x$, such that
	\[
	\p\left(V_n(a)>x \right)\leq \frac{C_1n^{1-q/3}}{x^{2q}}+2\mathrm{e}^{-C_2x^3}.
	\]
	It follows that
	\[
	\p\left(V_n(\lambda(s+n^{-1/3}x)+\epsilon_n)>\frac{\inf_{t\in[0,1]} |\lambda'(t)|d}{8\sup_{t\in[0,1]} |\lambda'(t)|}\right)
	\leq
	\frac{\tilde{C}_1n^{1-q/3}}{d^{2q}}+2\mathrm{e}^{-\tilde{C}_2d^3}.
	\]
	Similarly we can also bound the second probabilities in \eqref{eqn:probability_N_E} and \eqref{eqn:probability_N_E2}.
	Then the statement of the lemma follows immediately.
\end{proof}

\begin{proof}[Proof Theorem~\ref{theo:limit process}]
The proof is similar to the proof of Theorem~1.1 in\cite{kulikov-lopuhaa2006SPL}. 
{We briefly sketch the main steps.
Arguing as in the proof of Theorem~1.1 in\cite{kulikov-lopuhaa2006SPL}, it suffices} to show that for any compact $K\subset\R$, 
the process $\{A_n(t+sn^{-1/3}) : s\in K\}$ converges in distribution to the process $\{[D_{\R}Z_t](s):s\in K\}$ 
{on~$D(K)$,
the space of cadlag functions on~$K$, where~$Z_t$ is defined in~\eqref{def:Zt}.
By definition $A_n(t+sn^{-1/3})=[D_{I_{nt}}E_{nt}](s)$, for $s\in I_{nt}=[-tn^{1/3},(1-t)n^{1/3}]$, where $E_{nt}$ is defined in~\eqref{def:Zt}.}
To prove convergence in distribution, we show that for any bounded continuous function $g:D(K)\to\R$, 
\begin{equation}
\label{eqn:D_I_convergence}
\left|\E[g(D_{I_{nt}}E_{nt})]-\E[g(D_{\R}Z_t)]\right|\to 0.
\end{equation}
{To this end, we choose $d>0$ sufficiently large, such that $K\subset[-d/2,d/2]\subset [-d,d]=I$ and take~$n$ sufficiently large so that $I\subset I_{nt}$.
Then, similar to inequality (2.7) in~\citet{kulikov-lopuhaa2006SPL},} the triangular inequality yields
\begin{equation}
\label{eq:triangle}
\begin{split}
\left|\E[g(D_{I_{nt}}E_{nt})]-\E[D_{\R}Z_t]\right|&\leq \left|\E[g(D_{I_{nt}}E_{nt})]-\E[D_{I}E_{nt}]\right|\\
&\quad+\left|\E[g(D_{I}E_{nt})]-\E[D_{I}Z_t]\right|+\left|\E[g(D_{I}Z_{t})]-\E[D_{\R}Z_t]\right|.
\end{split}
\end{equation}
{In the same way as in~\citet{kulikov-lopuhaa2006SPL}, the three terms on the right hand side are shown to go to zero. 
For the last term on the right hand side of~\eqref{eq:triangle}, the argument is exactly the same and makes use of their Lemma~1.2.
The first term on the right hand side of~\eqref{eq:triangle} is bounded similar to their inequality~(2.9) and then uses Lemma~\ref{lemma:probability_N_W}.
For the second term on the right hand side of~\eqref{eq:triangle},} 
note that  from Lemma~\ref{lem:Lemma1.1SPL}, it follows that
\[
Z_{nt}(s)
=
n^{2/3}
\left(
\Lambda_n(t+sn^{-1/3})-\Lambda_n(t)
-
\left(
\Lambda(t+sn^{-1/3})-\Lambda(t)
\right)
\right)
+
\frac12\lambda'(t)s^2,
\]
converges in distribution to~$Z_t$. 
Therefore, because of the continuity of the mapping $D_I$, we get
{$|\E[h(D_{I}Z_{nt})]-\E[h(D_{I}Z_t)]|\to0$,
for any $h:D(I)\to\R$ bounded and continuous.
Moreover, 
we now have} $E_{nt}(s)=Z_{nt}(s)+n^{2/3}\Lambda_n(t)+\lambda(t)sn^{1/3}+R_{nt}(s)$,
where
\[
R_{nt}(s)
=
n^{2/3}
\left(
\Lambda(t+sn^{-1/3})-\Lambda(t)
-
\lambda(t)sn^{-1/3}-\frac12\lambda'(t)s^2n^{-2/3}
\right).
\]
{Similar to the argument leading up to (2.11) in~\citet{kulikov-lopuhaa2006SPL}, from the continuity of $D_I$, its invariance under addition of linear functions, and continuity of $\lambda'$, 
it follows that $|\E[g(D_{I}Z_{nt})]-\E[g(D_{I}E_{nt})]|\to 0$.
This establishes~\eqref{eqn:D_I_convergence} and finishes the proof.
}
\end{proof}

\begin{proof}[Proof of Lemma~\ref{lemma:4.1}]
	By a Taylor expansion, together with~\eqref{def:AnW}, we can write
	\[
	n^{2/3}\Lambda_n^W(t+n^{-1/3}s)=Y_{nt}(s)+L_{nt}^W(s)+R_{nt}^W(s),
	\]
	where
	$L_{nt}^W(s)=n^{2/3}\Lambda(t)+n^{1/6}W_n(L(t))+n^{1/3}\lambda(t)s$
	and
	\[
	R_{nt}^W(s)
	=
	n^{2/3}
	\left(
	\Lambda(t+n^{-1/3}s)-\Lambda(t)-n^{-1/3}
	\lambda(t)s-\frac{1}{2}n^{-2/3}\lambda'(t)s^2
	\right)
	=
	\frac{1}{6}n^{-1/3}\lambda''(\theta_1)s^3
	\]
	for some $|\theta_1-t|\leq n^{-1/3}|s|$.
	Then, from the assumptions (A1)-(A2), it follows that
	\[
	\sup_{|s|\leq \log n}\left|R_{nt}^W(s) \right|^p=O\left(n^{-p/3}(\log n)^{3p} \right),
	\]
	uniformly in $t\in(0,1)$.
	Similarly, we also obtain
	\[
	\begin{split}
	n^{2/3}\Lambda_n^E(t+n^{-1/3}s)
	&=
	n^{2/3}\Lambda_n^W(t+n^{-1/3}s)+n^{1/6}\left(E_n(t+n^{-1/3}s)-B_n(L(t+n^{-1/3}s))\right)\\
	&\quad-
	n^{1/6}\zeta_n\left(L(t)+L'(t)n^{-1/3}s\right)\\
	&\qquad-
	n^{1/6}\zeta_n\left(L(t+n^{-1/3}s)-L(t)-L'(t)n^{-1/3}s \right)\\
	&=Y_{nt}(s)+L_{nt}^E(s)+R_{nt}^E(s),
	\end{split}
	\]
	where
	$L_{nt}^E(s)=L_{nt}^W(s)-n^{1/6}\zeta_nL(t)-n^{-1/6}\zeta_nL'(t)s$
	and
	\[
	R_{nt}^E(s)
	=
	R_{nt}^W(s)+n^{1/6}\left(E_n(t+n^{-1/3}s)-B_n(L(t+n^{-1/3}s))\right)
	-
	\frac{1}{2}n^{-1/2}\zeta_nL''(\theta_2)s^2,
	\]
	for some $|\theta_2-t|\leq n^{-1/3}|s|$.
	Let $S_n=\sup_{s\in[0,1]}|E_n(s)-B_n(L(s))|$.
	From assumption (A2) we have
	$\p(S_n>n^{-1/2+1/q}x)\leq C_qx^{-q}$
	and it follows that
	\begin{equation}
	\label{eq:expec Sn}
	\begin{split}
	\E\left[S_n^p \right]
	&=
	\int_0^\infty \p\left(S_n^p\geq x \right)\,\dd x
	=
	p\int_0^\infty y^{p-1}\p\left(S_n\geq y \right)\,\dd y\\
	&=
	pn^{-p/2+p/q}\int_0^\infty x^{p-1}
	\p\left(S_n\geq n^{-1/2+1/q}x\right)\,\dd x\\
	&\leq
	pn^{-p/2+p/q}\left\{\int_0^1x^{p-1}\,\dd x+C_q\int_1^\infty x^{p-1-q}\,\dd x\right\}
	=
	O\left(n^{-p/2+p/q}\right),
	\end{split}
	\end{equation}
	if $p<q$. Consequently
	$\E\left[\sup_{|s|\leq \log n}\left|R_{nt}^E(s) \right|^p \right]=O\left(n^{-p/3+p/q}\right)$.
\end{proof}

\begin{proof}[Proof of Lemma~\ref{lemma:4.2}]	
	Note that we can write $A^J_n(t)\1_{N_{nt}^J}=n^{2/3}[\mathrm{D}_{I_{nt}}\Lambda^J_n](t)\1_{N_{nt}^J}$.
	We have
	\[
	\E\left[A^J_n(t)^p\right]
	=
	n^{2p/3}\E\left[[\mathrm{D}_{I_{nt}}\Lambda^J_n](t)^p\right]
	+
	\E\left[\left(A^J_n(t)^p-n^{2p/3}[\mathrm{D}_{I_{nt}}\Lambda^J_n](t)^p\right)\1_{(N_{nt}^J)^c}\right].
	\]
	To bound the second term on the right hand side, first note that
	\begin{equation}
	\label{eq:bound A-D}
	\left|A^J_n(t)^p-n^{2p/3}[\mathrm{D}_{I_{nt}}\Lambda^J_n](t)^p\right|
	\leq
	2A^J_n(t)^p,
	\end{equation}
	because the LCM on $[0,1]$ always lies above the LCM over $I_{nt}$.
	Since $\Lambda$ is concave, we have that
	\[
	\begin{split}
	\left|\CM_{[0,1]}\Lambda_n^E-\Lambda_n^E\right|
	&\leq
	\left|\CM_{[0,1]}\Lambda_n^E-[\CM_{[0,1]}\Lambda\right|
	+
	\left|\Lambda_n^E-\Lambda\right|
	+
	\left|\CM_{[0,1]}\Lambda-\Lambda\right|\\
	&=
	\left|\CM_{[0,1]}\Lambda_n^E-[\CM_{[0,1]}\Lambda\right|
	+
	\left|\Lambda_n^E-\Lambda\right|
	\leq
	2\sup_{s\in[0,1]}\left|\Lambda_n^E(s)-\Lambda(s)\right|,
	\end{split}
	\]
	which means that 
	$0\leq A_n^E(t)^p\leq 2^p n^{2p/3}\sup_{s\in[0,1]}\left|\Lambda_n^E(s)-\Lambda(s)\right|^p$.
	Furthermore,
	\[
	0\leq A_n^W(t)^p\leq 2^p n^{2p/3}\left\{\Lambda(1)+n^{-1/2}\sup_{s\in[0,1]}|W_n(s)|\right\}^p.
	\]
	In contrast to~\cite{kulikov-lopuhaa2008} it is more convenient to treat both cases separately.
	For the case $J=E$, with~\eqref{eq:bound A-D}, we find that
	\[
	\E\left[\left(A^E_n(t)^p-n^{2p/3}[\mathrm{D}_{I_{nt}}\Lambda^E_n](t)^p\right)\1_{(N_{nt}^E)^c}\right]
	\leq
	2^{p+1}n^{2p/3}
	\E\left[
	\sup_{s\in[0,1]}\left|\Lambda_n^E(s)-\Lambda(s)\right|^p
	\1_{(N_{nt}^J)^c}
	\right],
	\]
	where
	\[
	\sup_{s\in[0,1]}\left|\Lambda_n^E(s)-\Lambda(s)\right|^p
	\leq
	2^p
	\left\{
	\sup_{s\in[0,1]}\left|\Lambda_n^E(s)-\Lambda(s)-n^{-1/2}W_n(L(s))\right|^p
	+
	n^{-p/2}
	\sup_{s\in[0,1]}\left|W_n(L(s))\right|^p
	\right\}.
	\]
	For the first term on the right hand side we get with H\"older's inequality
	\[
	\begin{split}
	&
	n^{2p/3}
	\E
	\left[
	\sup_{s\in[0,1]}\left|\Lambda_n^E(s)-\Lambda(s)-n^{-1/2}W_n(L(s))\right|^p\1_{(N_{nt}^E)^c}
	\right]\\
	&\quad\leq
	n^{2p/3}
	\E
	\left[
	\sup_{s\in[0,1]}\left|\Lambda_n^E(s)-\Lambda(s)-n^{-1/2}W_n(L(s))\right|^{p\ell}
	\right]^{1/\ell}
	\mathbb{P}
	\left(
	(N_{nt}^E)^c
	\right)^{1/\ell'}\\
	&\quad=
	n^{2p/3}
	O(n^{-p+p/q})O\left(n^{1-q/3}(\log n)^{-2q}+\text{e}^{-C(\log n)^3} \right)^{1/\ell'},
	\end{split}
	\]
	for any $\ell,\ell'>1$ such that $1/\ell+1/\ell'=1$,
	according to~\eqref{eq:expec Sn} and Lemma~\ref{lemma:probability_N_W}.
	When $q>6$, then the right hand side is of the order $o(n^{-1/6})$.
	For the second term, with H\"older's inequality
	\[
	n^{2p/3}
	\E
	\left[
	\sup_{s\in[0,1]}\left|W_n(L(s))\right|^p
	\1_{(N_{nt}^E)^c}
	\right]
	\leq
	n^{2p/3}\E
	\left[
	\sup_{s\in[0,L(1)]}\left|W_n(s)\right|^{p\ell}
	\right]^{1/\ell}
	\mathbb{P}
	\left(
	(N_{nt}^E)^c
	\right)^{1/\ell'}
	\]
	for any $\ell,\ell'>1$ such that $1/\ell+1/\ell'=1$.
	Since all moments of $\sup_{s\in[0,L(1)]}\left|W_n(s)\right|$ are finite,
	it follows from Lemma~\ref{lemma:probability_N_W} that the right hand side is of the order
	\[
	n^{2p/3}\E
	\left[
	\sup_{s\in[0,1]}\left|W_n(L(s))\right|^p
	\1_{(N_{nt}^E)^c}
	\right]
	\leq
	n^{2p/3}O\left(n^{1-q/3}(\log n)^{-2q}+\text{e}^{-C(\log n)^3} \right)^{1/\ell'}.
	\]
	Hence, because $q>6$ and $p<2q-7$, it follows that
	$|A^E_n(t)^p-n^{2p/3}[\mathrm{D}_{I_{nt}}\Lambda^E_n](t)|=o(n^{-1/6})$.
	
	Next, consider the case $J=W$.
	Then with~\eqref{eq:bound A-D} and Cauchy-Schwarz, we find
	\[
	\begin{split}
	&
	\E\left[\left(A^W_n(t)^p-n^{2p/3}[\mathrm{D}_{I_{nt}}\Lambda^W_n](t)^p\right)\1_{(N_{nt}^W)^c}\right]\\
	&\leq
	2^{p+1}n^{2p/3}
	\left\{
	\E
	\left[
	\left(
	\Lambda(1)+n^{-1/2}\sup_{s\in[0,1]}|W_n(s)|
	\right)^{2p}
	\right]
	\right\}^{1/2}
	\left\{
	\mathbb{P}\left(
	(N_{nt}^W)^c
	\right)
	\right\}^{1/2}.
	\end{split}
	\]
	Again using that all moments of $\sup_{s\in[0,L(1)]}\left|W_n(s)\right|$ are finite,
	according to Lemma~\ref{lemma:probability_N_W}, the right hand side is of the order
	$n^{2p/3}O(\mathrm{e}^{-C(\log n)^3})=o(n^{-1/6})$.
	It follows that for $J=E,W$,
	\[
	\E\left[A^J_n(t)^p\right]
	=
	n^{2p/3}
	\E\left[[\mathrm{D}_{I_{nt}}\Lambda^J_n](t)^p\right]+o\left( n^{-1/6}\right).
	\]
	Moreover, Lemma \ref{lemma:4.1} implies that
	$n^{2/3}[\mathrm{D}_{I_{nt}}\Lambda^J_n](t)=[\mathrm{D}_{H_{nt}}Y_{nt}](0)+\Delta_{nt}$,
	where $\Delta_{nt}=[\mathrm{D}_{H_{nt}}(Y_{nt}+R^J_{nt})](0)-[\mathrm{D}_{H_{nt}}Y_{nt}](0)$.
	From Lemma~\ref{lemma:4.1}, we have
	\begin{equation}
	\label{eqn:moments_Delta}
	\E|\Delta_{nt}|^p
	\leq
	2^p
	\E\left[\sup_{|s|\leq \log n}\left|R_{nt}^J(s) \right|^p \right]
	=
	O\left(n^{-p/3+p/q}\right).
	\end{equation}
	Then as in Lemma 4.2 in \cite{kulikov-lopuhaa2008}, one can show that
	\[
	\begin{split}
	\E\left[A^J_n(t)^p\right]
	&=
	\E\left[[\mathrm{D}_{H_{nt}}Y_{nt}(0)^p\right]+\epsilon_{nt}+o\left( n^{-1/6}\right)\\
	&=
	\E\left[[\mathrm{D}_{H_{nt}}Y_{nt}(0)^p\right]+O\left(n^{-1/3+1/q}(\log n)^{2p-2}\right)+o\left( n^{-1/6}\right)\\
	&=
	\E\left[[\mathrm{D}_{H_{nt}}Y_{nt}(0)^p\right]+o\left( n^{-1/6}\right).
	\end{split}
	\]
	This finishes the proof.
\end{proof}

\begin{proof}[Proof of Lemma~\ref{lemma:4.4}]
The proof is exactly the same as the one for Lemma~4.4 in~\cite{kulikov-lopuhaa2008}. 
Define $J_{nt}=[n^{1/3}\left(L(a_{nt})-L(t)\right)/L'(t),n^{1/3}\left(L(b_{nt})-L(t) \right)/L'(t)]$,
where $a_{nt}=\max(0,t-n^{-1/3}\log n)$ and $b_{nt}=\min(1,t+n^{-1/3}\log n)$.
Furthermore, here we take
\[
\phi_{nt}(s)
=
\frac{n^{1/3}\left(L(t+n^{-1/3}s)-L(t) \right)}{L'(t)}.
\]
{As in the proof of Lemma~4.4 in~\cite{kulikov-lopuhaa2008},
it follows that $1-\alpha_n\leq \phi_{nt}(s)/s\leq 1+\alpha_n$, for 
$s\in H_{nt}$, the interval from Lemma~\ref{lemma:4.2}, 
and $\alpha_n=C_1n^{-1/3}\log n$, with $C_1>0$ only depending on $L'$.
Let $Z_t$ be the process in~\eqref{def:Zt}. 
Then}
\[
(Z_t\circ \phi_{nt})(s)=\tilde{Y}_{nt}+\frac{1}{2}\lambda'(t)s^2\left(\frac{\phi_{nt}(s)^2}{s^2}-1 \right), 
\]
where $\tilde{Y}_{nt}$ is defined in \eqref{def:tilde_Y}. 
{Lemma 4.3} in~\cite{kulikov-lopuhaa2008}, {then} allows us to approximate the moments of $[D_{H_{nt}}\tilde{Y}_{nt}](0)$ by the  moments of $[D_{J_{nt}}Z_t](0)$. 
{Completely similar to the proof of Lemma~4.4 in~\cite{kulikov-lopuhaa2008},
the result now follows from Lemma~\ref{lemma:4.2} and Brownian scaling.}
\end{proof}

\begin{proof}[Proof of Lemma~\ref{lemma:4.5}]
Let $I_{nt}$ and $N_{nt}^J$ be as in Lemma~\ref{lemma:4.2} and define $K_{nt}=N^E_{nt}\cap N^W_{nt}$. 
{Then
\begin{equation}
\label{eq:decomposition Lemma45}
\begin{split}
\E\left|
A_n^E(t)^p-A_n^W(t)^p
\right|
&=
n^{2p/3}
\E\left|
[D_{I_{nt}}\Lambda_n^E](t)^p-[D_{I_{nt}}\Lambda_n^W](t)^p
\right|
\1_{K_{nt}}\\
&\quad+
\E\left|
A_n^E(t)^p-A_n^W(t)^p
\right|
\1_{K_{nt}^c}.
\end{split}
\end{equation}
We bound the two terms on the right hand side, following the same line of reasoning as in Lemma~4.5 in~\cite{kulikov-lopuhaa2008}.}
Using that according to Lemma~\ref{lemma:probability_N_W},
\[
\p(K_{nt}^c)\leq \p\left((N^E_{nt})^c\right)+\p\left((N^W_{nt})^c\right)=O\left(n^{1-q/3}(\log n)^{-2q}+\mathrm{e}^{-C(\log n)^3} \right),
\]
the second term on the right hand side of~\eqref{eq:decomposition Lemma45} is of the order
$O\left(\p(K_{nt}^c)^{1/2}\right)=o\left( n^{-1/6}\right)$,
because $q>6$. 
On the other hand, the first term on the right hand side of~\eqref{eq:decomposition Lemma45} can be bounded by
\[
p\left\{\E\left[\left(A^E_n(t)^{p-1}+A^W_n(t)^{p-1}\right)^2\right] \right\}^{1/2}\left\{\E\left[\left(\sup_{|s|\leq \log n}|R^E_{nt}|+ \sup_{|s|\leq \log n}|R^W_{nt}|\right)^2\right]\right\}^{1/2},
\]
where the right hand side is of the order $O\left(n^{-1/3+1/q}\right)=o\left( n^{-1/6}\right)$,
according to Lemmas~\ref{lemma:4.1} and~\ref{lemma:4.4}.

In the same way, we have
$\E\left[\left|A_n^E(t)-A^W_n(t)\right|^p\1_{K_{nt}^c} \right]=O\left(\p(K_{nt}^c)^{1/2}\right)=o\left( n^{-1/6}\right)$
and
\[
n^{2p/3}\E\left[\left|[\mathrm{D}_{I_{nt}}\Lambda^E_n](t)-[\mathrm{D}_{I_{nt}}\Lambda_n^W](t)\right|^p\1_{K_{nt}} \right]
\leq
\E\left[\left(\sup_{|s|\leq \log n}|R^E_{nt}|+ \sup_{|s|\leq \log n}|R^W_{nt}|\right)^p\right]
\]
which is of the order $O\left(n^{-p/3+p/q}\right)=o\left( n^{-1/6}\right)$,
according to Lemma~\ref{lemma:4.1}.
\end{proof}

\begin{proof}[Proof of Lemma~\ref{lemma:variance_brownian_version}]
	The proof is completely similar to the proof of Lemma~4.7 in \cite{kulikov-lopuhaa2008}.
	For $t\in(0,1)$ fixed, and $t+c_2(t)sn^{-1/3}\in(0,1)$, let
	$\zeta_{nt}(s)=c_1(t)A_n^W(t+c_2(t)sn^{-1/3})$,
	where $A_n^W$ is defined in~\eqref{def:AnW} and $c_1(t)$ and $c_2(t)$ are defined in~\eqref{def:c1c2}.
	According to Theorem~\ref{theo:limit process}, $\zeta_{nt}$ converges in distribution to $\zeta$, 
	as defined in~\eqref{def:Z and zeta}.
	As in the proof of Lemma~4.7 in \cite{kulikov-lopuhaa2008}, Lemma~\ref{lemma:4.4} yields that,
	for $s$, $t$, and $k$ fixed, the sequence $\zeta_{nt}^W(s)^k$ is uniformly integrable, so that 
	the moments of $(\zeta_{nt}^W(0)^k,\zeta_{nt}^W(s)^k)$ converge to the corresponding moments of
	$(\zeta(0)^k,\zeta(s)^k)$.
	
	Furthermore, the process $\{A_n^W(t):t\in(0,)\}$ is strong mixing, i.e.,
	for $d>0$,
	\begin{equation}
	\label{eq:mixing}
	\sup
	\left|
	P(A\cap B)-P(A)P(B)
	\right|
	=
	\alpha_n(d)=48\mathrm{e}^{-Cnd^3}
	\end{equation}
	where $C>0$ only depends on $\lambda$ and $L$ from (A2),
	and where the supremum is taken over all sets $A\in\sigma\{A_n^W(s):0\leq s\leq t\}$
	and $B\in\sigma\{A_n^W(s):t+d\leq s < 1\}$.
	This can be obtained by arguing completely the same as in the proof of Lemma~4.6 in~\cite{kulikov-lopuhaa2008}.
	The rest of the proof is the same as that of  Lemma~4.7 in \cite{kulikov-lopuhaa2008}.
\end{proof}
\begin{proof}[Proof of Theorem \ref{theo:L_p_lcm}]
	The proof is completely similar to the proof of Theorem~2.1 in~\cite{kulikov-lopuhaa2008},
	by using the method of big-blocks small-blocks and the exponential decreasing mixing function 
	$\alpha_n$ from~\eqref{eq:mixing}.
\end{proof}

\section*{References}
\bibliography{shapeconstrained-estimationSPL}

\end{document}